\documentclass[12]{article}

\usepackage{graphicx}
\usepackage{amsthm, amsmath, amssymb,latexsym,caption}
\usepackage{verbatim}
\usepackage{mathrsfs}

\def\N{\mathbb{N}}
\def\tr{\textrm}
\def\disj{0}
\def\mathscr{}
\def\ind{\mathrm{ind}}
\def\dcrit{\tr{d}_{\tr{crit}}} 
\def\dyckm{\downarrow_{m}}
\def\dycku{\uparrow}
\def\dyckk{\downarrow_{k-1}}
\def\dycko{\downarrow}

\DeclareMathOperator{\minl}{minleaf}
\newtheorem{thm}{Theorem}
\newtheorem{lem}[thm]{Lemma}
\newtheorem{prop}[thm]{Proposition}

\begin{document}
\title{The density Tur\'an problem for hypergraphs}
\author{Adam Sanitt \and John Talbot}
\maketitle
\begin{abstract}
  \noindent
Given a $k$-graph $H$ a complete blow-up of $H$ is a $k$-graph $\hat{H}$ formed by replacing each $v\in V(H)$ by a non-empty vertex class $A_v$ and then inserting all edges between any $k$ vertex classes corresponding to an edge of $H$. Given a subgraph $G\subseteq \hat{H}$ and an edge $e\in E(H)$ we define the density $d_e(G)$ to be the proportion of edges present in $G$ between the classes corresponding to $e$.

The density Tur\'an problem for $H$ asks: determine the minimal value $\dcrit(H)$ such that any subgraph $G\subseteq \hat{H}$ satisfying $d_e(G)> \dcrit(H)$ for every $e\in E(H)$ contains a copy of $H$ as a transversal, i.e.~a copy of $H$ meeting each vertex  class of $\hat{H}$ exactly once.

  We give upper bounds for this hypergraph density Tur\'an problem that generalise the known bounds for the case of graphs due to Csikv\'ari and Nagy \cite{CsikvariNagy}, although our methods are different, employing an entropy compression argument.
\end{abstract}
\section{Introduction}
\label{sec:introduction}
The classical Tur\'an problem asks how many edges a graph or hypergraph $G$ can have if it does not contain a copy of a given forbidden subgraph $H$. The problem we consider is a variant known as the density Tur\'an problem (see Csikv\'ari and Nagy \cite{CsikvariNagy}). We consider subgraphs of blow-ups of a forbidden hypergraph $H$ (see below for formal definitions) and ask how dense this must be to guarantee a copy of the original hypergraph $H$.

In this paper we consider the general $k$-uniform hypergraph version of the problem. Let $H$ be an $k$-uniform hypergraph, or $k$-graph for short, with vertex set $V(H)=\{v_1,\ldots,v_h\}$ and edge set $E(H)\subseteq \binom{V(H)}{k}$. A $k$-graph $K$ is a \emph{subgraph} of the $k$-graph $H$ if and only if $V(K)\subseteq V(H)$ and $E(K)\subseteq E(H)$. The \emph{neighbourhood} of a vertex $v\in V(H)$ is \[\Gamma_H(v)=\{B\in \binom{V(H)}{k-1}\mid \{v\}\cup B\in E(H)\}.\] The \emph{degree} of a vertex $v$ is the size of its neighbourhood $|\Gamma_H(v)|$, while the \emph{maximum degree} of $H$ is $\Delta(H)=\max_{v\in V(H)}|\Gamma_H(v)|$. We also need the related concept of the \emph{maximum disjoint degree} of $H$: \[
\Delta_\disj(H)=\max_{v\in V(H)}\{j\mid \tr{there exist pairwise disjoint $B_1,\ldots, B_j\in \Gamma(v)$}\}.\]Note that $\Delta_\disj(H)\leq\Delta(H)$ with equality for all $2$-graphs and linear $k$-graphs.

A complete blow-up $\hat{H}$ is formed from $H$ by replacing each $v\in V(H)$ by a non-empty class $A_v$ containing $a_v$ vertices and then taking the edges of $\hat{H}$ to be all choices of $k$ vertices from any $k$ classes that correspond to an edge of $H$. More formally $\hat{H}$ has vertex set $V(\hat{H})=A_1\dot{\cup} \cdots\dot{\cup} A_h$ where  $A_i=\{w^i_1,\ldots,w_{a_i}^i\}\neq \emptyset$, and edge set \[ 
E(\hat{H})=\{w_{j_1}^{b_1}\cdots w_{j_k}^{b_k} \mid w_{b_1}\cdots w_{b_k}\in E(H),\ 1\leq j_i\leq a_i,\ 1\leq i\leq k\}.\]
If each vertex class has size $n$ then we call this the \emph{complete $n$-blow-up} of $H$ and denote it by $\hat{H}(n)$. We define a \emph{blow-up} of $H$ to be any subgraph $G\subseteq \hat{H}$ while an \emph{$n$-blow-up} is simply any subgraph $G\subseteq \hat{H}(n)$ with $V(G)=V(\hat{H}(n))$.  
 
An \emph{$H$-transversal} is a subgraph isomorphic to $H$ with exactly one vertex from each vertex class. We  say that a blow-up of $H$ is \emph{$H$-free} if it does not contain an $H$-transversal. 
We are interested in the question of when a blow-up of $H$ will contain an $H$-transversal. 

\begin{figure}[h]
  \centering
\includegraphics[width=11cm]{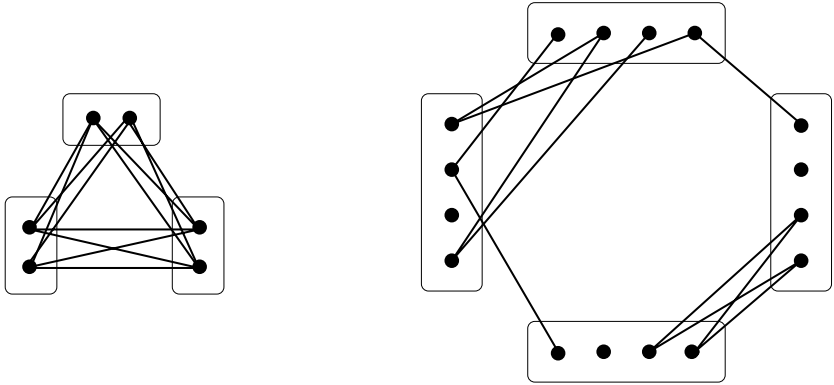}
\caption{$\hat{K_3}(2)$ the complete $2$-blow-up of $K_3$ and a $C_4$-free $4$-blow-up of $C_4$}
\end{figure}

If $G$ is a blow-up of $H$  and  $e\in E(H)$ we define $G[e]$ to be the $k$-partite subgraph of $G$ induced by $\cup_{v_i \in e}A_i$. We then define \[d_e(G)=\frac{|E(G[e])|}{\prod_{v_i\in e}a_i},\] which is simply the ordinary density of $G[e]$ and let $d(G)=\min_{e\in E(H)}d_e(G)$. Thus if $G$ is a blow-up of $H$ and $d(G)=d$ then every $k$-partite subgraph of $G$ formed from $k$ classes that correspond to an edge in $H$ has density at least $d$. 

The question we will consider is when does $d(G)> \delta$ imply that $G$ contains an $H$-transversal. We define the \emph{critical edge density} to be
\[
d_\tr{crit}(H)=\sup\{d(G)\mid \tr{$G$ is an $H$-free blow-up of $H$}\}.\]
Note that if $H$ is not connected then its critical edge density is simply the maximum of the critical edge densities of its components, so we will always assume that $H$ is connected.

\section{Previous work}
The first result in this area is due to Bondy et al.~\cite{BSTT} who considered the problem for triangles.
\begin{thm}[Bondy et al.~2006 \cite{BSTT}]
  The triangle $K_3$ has critical edge density $\dcrit(K_3)=\varphi\approx 0.618\ldots,$ the golden ratio.
\end{thm}
Later Nagy \cite{nagy} and Csikv\'ari and Nagy \cite{CsikvariNagy} gave exact results for trees and cycles as well as the following bound for the general graph version of the problem.
\begin{thm}[Csikv\'ari and Nagy 2012 \cite{CsikvariNagy}]\label{CN:thm}
  Let $H$ be a graph with maximum degree $\Delta$ and let $t(H)$ be the largest root of  its matching polynomial. Then the critical edge density satisfies
  \[
  \dcrit(H)\leq 1-\frac{1}{t(H)^2}.\]
  In particular,\begin{equation}\label{CN:eq}
  \dcrit(H)\leq 1-\frac{1}{4(\Delta-1)}.\end{equation}
 \end{thm} 
More recently Markstr\"om and Thomassen gave an exact answer for $K_{k+1}^{(k)}$. 
\begin{thm}[Markstr\"om and Thomassen 2019 \cite{Markstrom2019}]
  For $k\geq 3$, the complete $k$-uniform hypergraph of order $k+1$ has critical edge density
  \[
 \dcrit(K_{k+1}^{(k)})=\frac{k}{k+1}.\]\end{thm}

Our work is closest to Theorem \ref{CN:thm}. Using an entropy compression argument we derive an upper bound for the critical edge density of all $k$-graphs for $k\geq 2$. 
\begin{thm}\label{main:thm}
  Let $H$ be a $k$-graph of order $h$ with maximum degree $\Delta$ and maximum disjoint degree $\Delta_0$, then there exists a constant $\alpha=\alpha(k,\Delta_0)$ such that\[
  \dcrit(H)\leq 1-\frac{1}{\alpha\Delta}.\]
  Unless $H$ is $\Delta$-regular there also exists a constant $\beta(k,h)$, such that
\[
  \dcrit(H)\leq 1-\frac{1}{\beta(\Delta-1)}.\]
  Both $\alpha,\beta \leq k(k/(k-1))^{k-1}<ke$.   (The exact values of $\alpha,\beta$ above can be found by solving two related generalised Dyck-path counting problems which we discuss in Section \ref{gendyck:sec}.) \end{thm}
For $k=2$ we have $\beta\leq 4\cos^2\pi/(h+1)$ and so recover  (\ref{CN:eq}), the weaker of the two bounds from Theorem \ref{CN:thm} \cite{CsikvariNagy} in the case when $H$ is not regular.

We also have specific bounds for many complete $k$-graphs.
\begin{thm}\label{complete:thm}
  If $1\leq l<k$ then the complete $k$-graph of order $k+l$ satisfies
  \[
1-\frac{1}{\Delta}\leq  \dcrit(K_{k+l}^{(k)})\leq 1-\frac{1}{(l+1)\Delta}.\]
\end{thm} 
 
Interestingly, while these bounds are similar in form to the
previous bound for $2$-graphs, they are derived in a completely different
way. We use the entropy compression technique introduced by Moser and Tardos \cite{Moser}. 

The key ingredient is an algorithm which when given $G$, an $n$-blow-up of a $k$-graph $H$, searches for an $H$-transversal in $G$. This algorithm halts if and only if it finds such an $H$-transversal. As it runs, the algorithm consumes a sequence $(z_t)_{t=1}^s$ of integers and
maintains a record $(r_t)_{t=1}^s$ of its actions as well as a partial $H$-transversal $P_t$. We will show that using this record $(r_t)_{t=1}^s$ together with the final partial $H$-transversal $P_s$  it is possible to reconstruct the original sequence $(z_t)_{t=1}^s$ and so the algorithm can be viewed as a compression algorithm for integer sequences. We show that if $G$ is sufficiently dense and the search algorithm fails to halt, then this compression algorithm is simply too good to be true. This is an example of an \emph{entropy compression} argument (this terminology seems to have first been introduced by Tao \cite{Tao}).

In order to give the proof we will need some auxiliary results on generalised Dyck paths. The reader may skip ahead to Section \ref{proof:sec} for the proofs of Theorems \ref{main:thm} and \ref{complete:thm} and refer back to these results as necessary. 

We note that all our results concern upper bounds for the density Tur\'an problem.   Moving to a slightly more general setting, where one considers weighted hypergraphs, it is straightforward to generalise earlier results due to Bondy et al.~\cite{BSTT} for tripartite graphs  and due to Nagy \cite{nagy} for general graphs, showing that computing this critical edge density is in fact a finite optimisation problem. This in turn implies the following simple lower bound that we state without proof.
\begin{prop}\label{finite:prop} 
  If $H$ is a $k$-graph with maximum degree $\Delta$ then \[
  \dcrit(H)\geq 1-\frac{1}{\Delta}.\]
\end{prop}
\begin{proof}
  This is a simple generalisation to $k$-graphs of Corollary 3.8 \cite{nagy}.
\end{proof}
 
\section{Generalised Dyck paths}\label{gendyck:sec} 
For any integer $m\geq 1$, a \emph{partial $m$-Dyck path} is a path in the upper half-plane of the 2-dimensional integer lattice starting at $(0,0)$ using steps $\dycku\,\,=(1,1)$, a \emph{rise}, and $\dyckm\,\,=(1,-m)$, an \emph{$m$-fall}. The $y$-coordinate of any point on the path is known as its \emph{level}.  The \emph{height} of a path is the maximum level reached.  If the path ends on the horizontal axis, i.e.~at level zero, then it is called a \emph{full $m$-Dyck path}. The \emph{length} of the path is the number of steps. A longest sequence of consecutive $m$-falls in a partial $m$-Dyck path is called a \emph{maximum descent}.

\begin{figure}[h]
  \centering
\includegraphics[width=12cm]{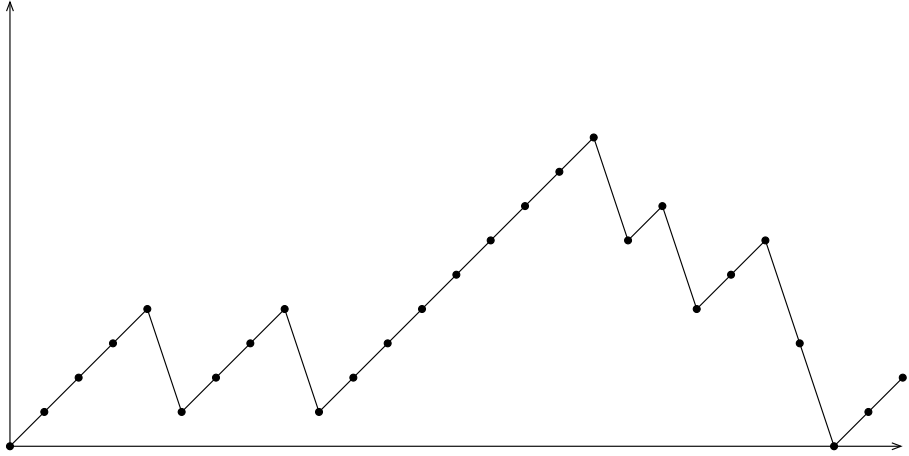}
\caption{A partial 3-Dyck-path of height 9, length 26 and max descent $2$.}
\end{figure}

We will be interested in counting partial $m$-Dyck paths of bounded height and bounded maximum  descent.

Given integers $h,m,s,d,l\geq 0$  let $\mathcal{D}_m(s,l,h,d)$ denote the set of partial $m$-Dyck paths ending at  $(s,l)$, bounded by height $h$ and with maximum descent at most $d$. We will also be interested in paths with no restriction on height or maximum descent in which case we will replace the argument $h$ or $d$ by $\cdot$. 

\begin{lem}\label{full:lem}
  Given integers $h\geq m$ and $s,d,l\geq 1$ such that $\max\{d,m\}\geq 2$
  there exists $t\leq \lfloor \frac{s+4h}{m+1}\rfloor$ such that  \[|\mathcal{D}_m(s,l,h,d)|\leq |\mathcal{D}_m(t(m+1),0,h,d)|.\] 
\end{lem} 
\begin{proof} We  define an injective  mapping from $\mathcal{D}_m(s,l,h,d)$ to $\mathcal{D}_m(t(m+1),0,h,d)$, for some $t=t(m,l)$. We do this by describing a fixed extension of each path in the domain that depends only on $m$ and $l$  using at most  $4h$ additional steps. In each case this yields a full $m$-Dyck path and hence  has length $t(m+1)$ for some $t\leq \lfloor \frac{s+4h}{m+1}\rfloor$. Moreover the new path is constructed so that it still has height at most $h$ and maximum descent at most $d$.

Fix a partial $m$-Dyck path from $\mathcal{D}_m(s,l,h,d)$ ending at $(s,l)$. Starting from $(s,l)$, we extend this path as follows. If $m\geq2 $ first add $\dycku^{h-l}$ i.e.~$h-l$ rises. Next let $j=\lfloor h/(m-1)\rfloor$ and add the following $2j$ steps $(\dyckm\dycku)^j$. This takes us to level $0\leq l_0\leq m-2$. Either $l_0=0$ and we are done or add $\dycku^{m-l_0}\dyckm$ to give a full $m$-Dyck path with at most $4h$ additional steps.

For $m=1$ we start by adding $\dycku^{h-l}$. Next we use the fact that $d\geq 2$ to add $(\dycko^2\dycku)^{h-2}\dycko^2$ which gives a full $1$-Dyck path with at most  $4h$ additional steps.

In each case note that the extended path never exceeds height $h$. Moreover the maximum descent in each extended path is still at most $d$.
  \end{proof}
\begin{lem}\label{bound1:lem} 
  Given integers $m,t,d\geq 1$ such that $\max\{d,m\}\geq 2$. Let $\phi(x)=\sum_{i=0}^dx^{mi}$ and let $\tau$ be the unique positive solution of $\phi(x)=x\phi'(x)$. If $\alpha(m,d)=(\phi'(\tau))^m$ then there exists a constant $c_{a}=c_{a}(m,d)$ such that 
  \[|\mathcal{D}_m(t(m+1),0,\cdot,d)|\leq c_a\alpha^{t}.\]        
Moreover $\alpha\leq \gamma_m=(m+1)(1+1/m)^m$.
\end{lem}    
\begin{proof}  The fact that $\alpha\leq \gamma_m$ follows by noting that $\gamma_m=\alpha(m,\infty)$ (i.e.~the value of $\alpha$ obtained by setting $d=\infty$ in the sum $\phi(x)$).
  
We use Lemma 8 from Esperet and Parreau \cite{EsperetParreau}, that in turn uses the work of Drmota \cite{Drmota}.

Counting full $m$-Dyck paths with $t$ $m$-falls and maximum descent $d$ is the same as counting $1$-Dyck paths of length $2tm$ with all descents from the set $E=\{m,2m,\ldots,dm\}$. (Simply replace each $m$-fall by $m$ $1$-falls.)

Lemma 8 \cite{EsperetParreau} can now be applied, with $\phi_E(x)=\sum_{i=0}^dx^{mi}$, to give a constant $c_E$ such that the number of such paths is at most $c_E\alpha^t$. So we can take $c_{a}(m,d)=c_E$.
\end{proof}      
\begin{lem}\label{bound2:lem}
  Given integers $h,m,t\geq 1$ let $\beta(m,h)$ be the reciprocal of the smallest root of \[\sum_{i=0}^{\lfloor (h+1)/(m+1)\rfloor} (-x)^i\binom{h-mi+1}{i}.\] There exists a constant $c_b=c_b(m,h)$ such that
  \[|\mathcal{D}_m(t(m+1),0,h,\cdot)|\leq c_b\beta^{t}.\]
Moreover $\beta(1,h)=4\cos^2\pi/(h+2)$.
\end{lem}
\begin{proof} The enumeration of $m$-Dyck paths of bounded height $h$ is a special case of enumerating $m$-Dyck paths with weights $\alpha_i$ associated to descents from different levels. More precisely, associate to each step of an $m$-Dyck path a weight of $\alpha_i$ for a descent from height $i$ and a weight of $1$ for a rise. The weight of the path is then the product of the weights of its steps. Setting $\alpha_i=1$ for $0\leq i\leq h$ and $\alpha_i=0$ for $i>h$, the sum of weighted $m$-Dyck paths of a given length is simply the number of such paths of height bounded by $h$. This problem is considered by P\'etr\'eolle et al.~\cite{petreolle2018lattice}. For $j\geq -1$ we define\[
  g_{j}(x)=\sum_{i=0}^{\lfloor(h-j)/(m+1)\rfloor}(-x)^i\binom{h-j-mi}{i}.\]
  It is easy to check that these polynomials satisfy the recurrence:\[
  g_k(x)-g_{k-1}(x)=\left\{\begin{array}{ll} xg_{k+m}(x),& 0\leq k\leq h-m,\\ 0, & k>h-m.\end{array}\right.\] So Proposition 2.3 \cite{petreolle2018lattice} implies that $g_0(x)/g_{-1}(x)$ is the ordinary generating function for $m$-Dyck paths of height bounded by $h$.
  
  By the Cauchy--Hadamard theorem the asymptotic growth rate of the coefficients is the reciprocal of the  radius of convergence of the generating function. Since the generating function is the ratio of polynomials the radius of convergence is determined by the smallest root of $g_{-1}(x)$ and so the result follows.

  The fact that $\beta(1,h)=4\cos^2\pi/(h+2)$ can be found in de Bruijn et al. \cite{BKR}.
\end{proof} 
\section{Proof of main results}
\label{proof:sec}
\begin{proof}[Proof of Theorem \ref{main:thm}:]
Let $H$ be a $k$-graph with vertex set $V(H)=[h]:=\{1,2,\ldots,h\}$. Let $G$ be an $n$-blow-up of $H$ with vertex classes $A_1,\ldots,A_h$, where $A_i=\{w_1^i,\ldots,w_n^i\}$. Let $\Delta=\Delta(H)$ and $\Delta_0=\Delta_0(H)$.
Suppose, for a contradiction, that $G$ is $H$-free and has density \[d(G)\geq 1-\frac{1}{\alpha\Delta}+\epsilon,\] for some $\epsilon>0$ and where $\alpha=\alpha(k-1,\Delta_0)$ from Lemma \ref{bound1:lem}.

Define a projection map, $\pi_H:V(G)\to [h]$, by $\pi_H(w_j^i)=i$. We also define an index map $\ind_i:A_i\to[n]$ by $\ind_i(w_j^i)=j$. Given $P\subseteq V(G)$ and $e\in E(H)$ we define $P(e)=P\cap\cup_{i\in e}A_i$, this is the restriction of $P$ to those vertex classes of $G$ corresponding to the edge $e$.

We say $P\subseteq V(G)$ is a \emph{partial $H$-transversal} if and only if (i) $|P\cap A_i|\leq 1$ for $1\leq i \leq h$ and (ii) for every $e\in E(H)$ if $e\subseteq \pi_H(P)$ then $P(e)\in E(G)$. (Condition (i) ensures that no vertex class has more than one representative, while (ii) ensures that the subgraph induced by $P$ has all edges that are required.)
Note that if $P$ is a partial $H$-transversal then $\pi_H(P)$ is precisely the set of vertices of $H$ that are represented in $P$.

Consider running Algorithm (A) below.
\begin{figure}
  \begin{minipage}[c]{\textwidth}
    \textbf{Algorithm (A)}
  \begin{tabbing}
  \texttt{input:} \= $H$, $G$ an $n$-blow-up of $H$, $(z_{t})_{t=1}^\infty\in\{1,2,\ldots,n\}^\N$.\\
  \texttt{initialize:}  $P_0\leftarrow\emptyset$, $t\leftarrow1$.\\
  \texttt{wh}\=\texttt{ile ($P_{t-1}$ is not an $H$-transversal)} \\\> $i_t\leftarrow\min[h]\setminus \pi_H(P_{t-1})$.\\
  \>\texttt{wh}\=\texttt{ile ($i_t\not\in\pi_H(P_{t-1})$)}\\
 \>\> $P_t\leftarrow P_{t-1}\cup\{w^{i_t}_{z_t}\}$.\\
\>\>\texttt{if}\= \texttt{ ($P_t$ is a partial $H$-transversal)}\\\>\>\>\texttt{then }  $r_t\leftarrow1$\\ 
\>\> \texttt{el}\=\texttt{se}\\  
  \>\>\>\texttt{choose $e\in E(H)$ such that $e\subseteq \pi_H(P_t) $ and $P_t(e)\not\in E(G)$}\footnote{If more than one choice is available then select any.}  \\ 
  
  \>\>\>  $r_t\leftarrow P_t(e)$\\\>\>\> $P_t\leftarrow P_t\setminus P_t(e)$\\\>\>\> $i_{t+1}\leftarrow i_t$\\
  \>\>  $t\leftarrow t+1$\\
  \>\texttt{continue}\\
  \texttt{continue}\\
  \texttt{output $P_{t-1}$ and halt.}
\end{tabbing}
\end{minipage}
  \end{figure}

First note that Algorithm (A) does indeed make sense as a search algorithm for an $H$-transversal in $G$. At time $t$ it considers $i_t$, the smallest vertex of $H$ that is not currently represented in the partial $H$-transversal $P_{t-1}$. It then uses the next integer in the sequence $(z_{t})_{t=1}^\infty$  to select a vertex  $w_{z_t}^{i_t}\in A_{i_t}$. If adding this vertex to $P_{t-1}$ gives a partial $H$-transversal then the algorithm records this success by setting $r_t=1$ and continues to the next unrepresented vertex in $V(H)$. However if adding this vertex creates a set that is no longer a partial $H$-transversal then there must be an edge  $e\in E(H)$ such that the corresponding edge $P_t(e)$ is missing from $E(G)$. In this case the algorithm chooses one such edge $e\in E(H)$ and records the fact that it is missing from $E(G)$ by setting $r_t=P_t(e)$. The vertices in $P_t(e)$  are then removed from $P_t$ and at time $t+1$ the algorithm again tries to add a vertex to the same vertex class.

Let $(r_t)_{t=1}^s$ denote the record produced by the algorithm up to time $s$. When we refer to $P_t$ we mean the set $P_t$ at the end of the $t^\tr{th}$ iteration of the algorithm, i.e.~at the moment that $t\leftarrow t+1$. We claim that given $(P_s,(r_t)_{t=1}^s)$ we can reconstruct $(z_t)_{t=1}^s$, the integer sequence up to time $s$.
 
First we use $(r_t)_{t=1}^s$ to reproduce the sequences $(i_t)_{t=1}^s$ and $(\pi_H(P_t))_{t=1}^s$. This follows by induction on $t$. Clearly $i_1=1$ and $\pi_H(P_1)=\{1\}$, so suppose now  that we have $(r_t)_{t=1}^s$ and $i_{t}$,  $\pi_H(P_t)$ are both known for some $1\leq t<s$. If $r_t=1$ then $i_{t+1}=\min[h]\setminus \pi_H(P_t)$ otherwise $i_{t+1}=i_t$. Using this we can obtain
\[
\pi_H(P_{t+1})=\left\{
\begin{array}{ll}
\pi_H(P_t)\cup \{i_{t+1}\},& \textrm{if }r_{t+1}=1,\\
\pi_H(P_t)\setminus \pi_H(r_{t+1}),&\textrm{otherwise.}
\end{array}\right.\]
We can now reconstruct both $(z_t)_{t=1}^s$ and $(P_t)_{t=1}^s$ using  $(i_t)_{t=1}^s$, $(\pi_H(P_t))_{t=1}^s$ and $(P_s,(r_t)_{t=1}^s)$. We use reverse induction on $t$.
Indeed we are given $P_s$ and if we have found $P_t$ for any  $t\leq s$ then 
\[
z_t=\left\{
\begin{array}{ll}
\ind_{i_t}(P_t\cap A_{i_t}),& \textrm{if }r_{t}=1,\\
\ind_{i_t}(r_t \cap A_{i_t}),&\textrm{otherwise.}
\end{array}\right.\]
Moreover having obtained $z_t$ we can then find $P_{t-1}$ since
\[
P_{t-1}=\left\{
\begin{array}{ll}
P_t\setminus \{w_{z_t}^{i_t}\},& \textrm{if }r_{t}=1,\\
P_t\cup r_t,&\textrm{otherwise.}
\end{array}\right.\]
Hence we can recover $(P_t)_{t=1}^s$ and $(z_t)_{t=1}^s$ as required.

Since $G$ is by assumption $H$-free, Algorithm (A) never halts irrespective of the integer sequence $(z_t)_{t=1}^s\in[n]^s$. This implies that there must be at least $n^s$ possibilities for $(P_s,(r_t)_{t=1}^s)$. 

We focus first on enumerating the possibilities for $(r_t)_{t=1}^s$. We form a modified version of this sequence that keeps track of the size of the partial $H$-transversal $P_t$. This modified sequence is a partial $(k-1)$-Dyck path (recall that $H$ is a $k$-graph) defined by \[
r^\circ_t=\left\{
\begin{array}{ll}
  \dycku,& r_t=1,\\
  \dyckk,& r_t=P_t(e). \end{array}\right.\] So $r_t^\circ$ simply records the change in the size of $|P_t|$ on the $t^\tr{th}$ iteration of the algorithm. Since $P_0=\emptyset$ and the algorithm never builds a complete $H$-transversal, $(r_t^\circ)_{t=1}^s$ is a partial $(k-1)$-Dyck-path of length $s$, with height bounded above by $h-1=|V(H)|-1$. (See Section \ref{gendyck:sec} for definitions.)

A sequence of repeated $(k-1)$-falls in this path corresponds to repeatedly removing edges from a single vertex in $H$ that meet only at this vertex, so there are never more than $\Delta_\disj$ such $(k-1)$-falls in a row. (Recall that $\Delta_\disj$ is the maximum disjoint degree of $H$.) Hence this path has maximum descent at most $\Delta_\disj $.

How many different sequences $(r_t)_{t=1}^s$ could give rise to the same path $(r_t^\circ)_{t=1}^s$? For each $r_t^\circ=\,\,\dycku$ there is a single choice for $r_t$, namely $r_t=1$. While if $r_t^\circ=\,\,\dyckk$ then there is an edge $e\in E(H)$ that contains the vertex $i_t$ such that $r_t=P_t(e)\not\in E(G)$. The number of possible choices for $e$ is at most the degree of $i_t$ in $H$ which is at most $\Delta$. Moreover, the number of choices for $P_t(e)$ given $e$ is at most $n^k-|G[e]|\leq n^k(1-d(G))$. Thus overall the number of choices for $P_t(e)$ is at most $\Delta(1-d(G))n^k$.

A path $(r_t^\circ)_{t=1}^s$ contains at most $s/k$ $(k-1)$-falls (since it has length $s$ and always remains in the upper half-plane) so at most $(\Delta(1-d(G))n^k)^{s/k}$ distinct original sequences $(r_t)_{t=1}^s$ can give rise to the same path.

Finally, since each $(r_t^\circ)_{t=1}^s\in\mathcal{D}_{k-1}(s,l,h-1,\Delta_0)$ for some $0\leq l\leq h-1$, Lemma \ref{full:lem} together with Lemma \ref{bound1:lem} imply that the number of different possible sequences $(r_t)_{t=1}^s$ is at most
\[
(\Delta(1-d(G))n^k)^{s/k}hc_{a}\alpha^{(s+4h)/k},\]
where $c_{a}=c_{a}(k-1,\Delta_0)$ and $\alpha=\alpha(k-1,\Delta_0)$.

Recall that we wanted to count the possibilities for $(P_s,(r_t)_{t=1}^s)$, which should be at least $n^s$ since this is the number of different integer sequences that can be reconstructed from this information. The number of possibilities for the final partial $H$-transversal $P_s$ is less than $(n+1)^h$ since $P_s$ consists of a choice of at most one vertex from each vertex class $A_i$.
Hence
\[
 hc_{a}(n+1)^h\alpha^{4h/k}(\Delta\alpha(1-d(G))n^k)^{s/k}\geq n^s.\]
But by assumption $1-d(G)\leq 1/\alpha\Delta-\epsilon$, so for $s,n$ large this is impossible. This proves the first inequality in the theorem.
 
The second inequality in Theorem \ref{main:thm}, for non-$\Delta$-regular $H$, follows from a simple variant of the method. We now use a tree to choose the vertex class under consideration at time $t$.

Given a connected $k$-graph $H$, we define the \emph{skeleton} of $H$ to be the 2-graph $H_2$ with vertex set $V(H)$ and $xy\in E(H_2)$ if and only if there is a hyperedge $e\in E(H)$ such that $x,y\in e$. Given a tree $T\subseteq H_2$ and a hyperedge $f\in E(H)$ that meets $T$ in a single leaf $v$ we define $T\oplus_v f$ to be the tree in $H_2$ formed from $T$ by adding each vertex $w\in f\setminus\{v\}$ as a leaf with parent $v$.  We define $\minl(T)$ to be the smallest leaf of $T$ (recall $V(H)=[h]$ is ordered).

Suppose, for a contradiction, that $G$ is $H$-free and has density \[d(G)\geq 1-\frac{1}{\beta(\Delta-1)}+\epsilon,\] for some $\epsilon>0$ and where $\beta=\beta(k-1,h-1)$ from Lemma \ref{bound2:lem}.

Consider running Algorithm (B) described below. The input we give is the same as to Algorithm (A) with the addition of a spanning tree $T$ of $H_2$. Recall that $V(H)=[h]$. Since $H$ is not $\Delta$-regular we may assume (by reordering $V(H)$ if needed) that the degree of vertex $h$ is at most $\Delta-1$.
\begin{figure}
\begin{minipage}[c]{\textwidth}
  \textbf{Algorithm (B)}
\begin{tabbing} 
  \texttt{input:} \= $H$, $T$ a spanning tree of $H_2$, $G$ an $n$-blow-up of $H$, $(z_{t})_{t=1}^\infty\in\{1,2,\ldots,n\}^\N$.\\
  \texttt{initialize:}  $P_0\leftarrow\emptyset$, $T_0\leftarrow T$, $t\leftarrow1$.\\
 \texttt{wh}\=\texttt{ile ($P_{t-1}$ is not an $H$-transversal)} \\\> $i_t\leftarrow\minl(T_{t-1})$.\\
 \> $P_t\leftarrow P_{t-1}\cup\{w^{i_t}_{z_t}\}$.\\
\>\texttt{if }\=\texttt{($P_t$ is a partial $H$-transversal)}\\
\>\>\texttt{th}\=\texttt{en }\\
\>\>\>  $r_t\leftarrow1$\\
\>\>\> $T_t\leftarrow T_{t-1}\setminus \{i_t\}$\\
\>\> \texttt{el}\=\texttt{se}\\   
  \>\>\>\texttt{choose $e\in E(H)$ such that $e\subseteq \pi_H(P_t) $ and $P_t(e)\not\in E(G)$}\footnote[2]{Again if more than one choice is available then select any.}   \\  
  \>\>\>  $r_t\leftarrow P_t(e)$\\\>\>\> $T_t\leftarrow T_{t-1}\oplus_{i_t}\pi_H(P_t(e))$\\
  \>\>\> $P_t\leftarrow P_t\setminus P_t(e)$\\
  \>  $t\leftarrow t+1$\\
  \texttt{continue}\\
  \texttt{output $P_{t-1}$ and halt.}
\end{tabbing}
\end{minipage}
\end{figure}

First note that Algorithm (B) does indeed make sense as a search algorithm for an $H$-transversal in $G$. At time $t$ it considers $i_t$, the smallest leaf of the current tree $T_{t-1}$. It then uses the next integer in the sequence $(z_{t})_{t=1}^\infty$  to select a vertex  $w_{z_t}^{i_t}\in A_{i_t}$. If adding this vertex to $P_{t-1}$ gives a partial $H$-transversal then the algorithm records this success by setting $r_t=1$ and deletes the leaf $i_t$ from $T_{t-1}$ to give the next tree $T_t$. However if adding this vertex creates a set that is no longer a partial $H$-transversal then there must be an edge  $e\in E(H)$, containing $i_t$, such that the corresponding edge $P_t(e)$ is missing from $E(G)$. In this case the algorithm chooses one such edge $e\in E(H)$ and records the fact that it is missing from $E(G)$ by setting $r_t=P_t(e)$. The vertices of $e$ (except $i_t$) are then added as leaves adjacent to $i_t$ in $T_{t-1}$ to give the next tree $T_t$, while the vertices in $P_t(e)$  are removed from $P_t$. (Note that as $\pi_H(P_{t-1})$ and $T_{t-1}$ are disjoint by construction, $e$ meets $T_{t-1}$ only at $i_t$ so this does indeed yield a tree.)

Let $(r_t)_{t=1}^s$ denote the record produced by the Algorithm (B) up to time $s$. As before, when we refer to $P_t$ we mean the set $P_t$  at the end of the $t^\tr{th}$ iteration of the algorithm, i.e.~at the moment that $t\leftarrow t+1$. We claim that given $(P_s,(r_t)_{t=1}^s)$ we can reconstruct $(z_t)_{t=1}^s$.
 
First we use $(r_t)_{t=1}^s$ to reproduce the sequences $(T_t)_{t=1}^s$ and $(\pi_H(P_t))_{t=1}^s$. This follows by induction on $t$. Clearly $i_1=\minl(T)$ and $\pi_H(P_1)=\{i_1\}$, so suppose $T_t$,  $\pi_H(P_t)$ are both known for some $1\leq t<s$. We have $i_{t+1}=\minl(T_t)$ so
\[
T_{t+1}=\left\{
\begin{array}{ll}
  T_t\setminus \{i_{t+1}\},& \textrm{if }r_{t+1}=1,\\
  T_{t}\oplus_{i_{t+1}}\pi_H(r_{t+1}),& \textrm{otherwise.} 
\end{array}\right.\]
While $\pi_H(P_{t+1})=V(H)\setminus T_t$.

Note that having found $\{T_t\}_{t=1}^s$  we have $i_t=\minl(T_{t-1})$ so we also know $(i_t)_{t=1}^s$.
We can now reconstruct both $(z_t)_{t=1}^s$ and $(P_t)_{t=1}^s$ using  $(i_t)_{t=1}^s$, $(\pi_H(P_t))_{t=1}^s$ and $(P_s,(r_t)_{t=1}^s)$. We use reverse induction on $t$.
Indeed we are given $P_s$ and if we have found $P_t$ for any  $t\leq s$ then 
\[
z_t=\left\{
\begin{array}{ll}
\ind_{i_t}(P_t\cap A_{i_t}),& \textrm{if }r_{t}=1,\\
\ind_{i_t}(r_t \cap A_{i_t}),&\textrm{otherwise.}
\end{array}\right.\]
Moreover having obtained $z_t$ we can then find $P_{t-1}$ since
\[
P_{t-1}=\left\{
\begin{array}{ll}
P_t\setminus \{w_{z_t}^{i_t}\},& \textrm{if }r_{t}=1,\\
P_t\cup r_t,&\textrm{otherwise.}
\end{array}\right.\]
Hence we can recover $(P_t)_{t=1}^s$ and $(z_t)_{t=1}^s$ as required.

Since $G$ is by assumption $H$-free, Algorithm (B) never halts irrespective of the integer sequence $(z_t)\in[n]^s$. Moreover, since we can reconstruct this integer sequence from $(P_s,(r_t)_{t=1}^s)$ and there are $n^s$ such sequences,  there must be at least $n^s$ possibilities for $(P_s,(r_t)_{t=1}^s)$. 

As before we form a modified version of this sequence that keeps track of the size of the partial $H$-transversal $P_t$. This modified sequence  is a partial $(k-1)$-Dyck path defined by \[
r^\circ_t=\left\{
\begin{array}{ll}
  \dycku,& r_t=1,\\
  \dyckk,& \textrm{otherwise.} \end{array}\right.\] So $r_t^\circ$ simply records the change in the size of $|P_t|$ on the $t^\tr{th}$ iteration of the algorithm. Since $P_0=\emptyset$ and the algorithm never builds a complete $H$-transversal, $(r_t^\circ)_{t=1}^s$ is a partial $(k-1)$-Dyck-path of length $s$, with height bounded above by $h-1=|V(H)|-1$. (See Section \ref{gendyck:sec} for definitions.)

How many different sequences $(r_t)_{t=1}^s$ could give rise to the same path $(r_t^\circ)_{t=1}^s$? For each $r_t^\circ=\,\,\dycku$ there a single choice for $r_t$, namely $r_t=1$. While if $r_t^\circ=\,\,\dyckk$ then there is an edge $e\in E(H)$ that contains the vertex $i_t$ such that $r_t=P_t(e)\not\in E(G)$. If $i_t=h$ then the number of possible choices for $e$ is at most $\Delta-1$ (since by assumption this vertex has degree less than $\Delta$). Otherwise, since  $i_t$ is a leaf of $T_{t-1}$, it has a parent $p$ in $T_{t-1}$. In this case the possible choices for $e$ are all edges such that $i_t\in e$ and $p\not\in e$. This is again at most $\Delta-1$ since at least one edge contains both vertices. The number of choices for $P_t(e)$ given $e$ is at most $n^k-|G[e]|\leq n^k(1-d(G))$. Thus overall the number of choices for $P_t(e)$ is at most $(\Delta-1)(1-d(G))n^k$.

A path $(r_t^\circ)_{t=1}^s$ contains at most $s/k$ $(k-1)$-falls (since it has length $s$ and always remains in the upper half-plane) so at most $((\Delta-1)(1-d(G))n^k)^{s/k}$ distinct original sequences can give rise to the same path.

Finally, since  $(r_t^\circ)_{t=1}^s\in\mathcal{D}_{k-1}(s,l,h-1,\cdot)$ for some $0\leq l\leq h-1$ so Lemma \ref{full:lem} and \ref{bound2:lem} imply that the number of different possible sequences $(r_t)_{t=1}^s$ is at most
\[
 hc_b\beta^{(s+4h)/k}((\Delta-1)(1-d(G))n^k)^{s/k},\]
where $c_b=c_b(k-1,h-1)$ and $\beta=\beta(k-1,h-1)$.

As before the number of possibilities for $P_s$ is less than $(n+1)^h$ since $P_s$ consists of a choice of at most one vertex from each vertex class $A_i$.
Hence counting possiblities for  $(P_s,(r_t)_{t=1}^s)$ we must have
\[
 hc_b(n+1)^h\beta^{4h/k}((\Delta-1)\beta(1-d(G))n^k)^{s/k}\geq n^s.\]
 But by assumption \[
 1-d(G)\leq \frac{1}{\beta(\Delta-1)}-\epsilon,\] so for $s,n$ large this is impossible.  This proves the second bound and completes the proof of the theorem.
\end{proof}  
\begin{proof}[Proof of Theorem \ref{complete:thm}]
 This follows easily using Algorithm (A) above and noting that for $H=K_{k+l}^{(k)}$ it is easy to count the exact number of partial $(k-1)$-Dyck paths of length $tk+k-1$ bounded by height $k+l-1$. Any such path starts with $k-1$ rises, then one of the next $(l+1)$ steps must be be a $(k-1)$-fall and then the path must return to level $k-1$ using rises. Hence the number of such paths is exactly  $(l+1)^t$. The result then follows as before.
  \end{proof}
\bibliography{joc_final}

\begin{thebibliography}{10}

\bibitem{BSTT}
John~Adrian Bondy, Jian Shen, St{\'{e}}phan Thomass{\'{e}}, and Carsten
  Thomassen.
\newblock Density conditions for triangles in multipartite graphs.
\newblock {\em Combinatorica}, 26(2):121--131, 2006.

\bibitem{BKR}
N.~Bruijn, D.~Knuth, and S.~Rice.
\newblock The average height of planted plane trees.
\newblock {\em Graph Theory and Computing}, page~10, 05 1971.

\bibitem{CsikvariNagy}
P.~Csikv\'ari and Z.~L. Nagy.
\newblock The density {T}ur{\'a}n problem.
\newblock {\em Combinatorics, Probability and Computing}, 21(4):531--553, 2012.

\bibitem{Drmota}
M.~Drmota.
\newblock {\em Random Trees: An Interplay Between Combinatorics and
  Probability}.
\newblock Springer Publishing Company, Incorporated, 1st edition, 2009.

\bibitem{EsperetParreau}
L.~Esperet and A.~Parreau.
\newblock Acyclic edge-coloring using entropy compression.
\newblock {\em European Journal of Combinatorics}, 34(6):1019 -- 1027, 2013.

\bibitem{Markstrom2019}
K.~Markstr{\"{o}}m and C.~Thomassen.
\newblock Partite {T}ur{\'a}n-densities for complete $r$-uniform hypergraphs on
  $r+1$ vertices.
\newblock {\em CoRR}, abs/1903.04270, 2019.

\bibitem{Moser}
Robin~A. Moser and G{\'{a}}bor Tardos.
\newblock A constructive proof of the general lovasz local lemma.
\newblock {\em CoRR}, abs/0903.0544, 2009.

\bibitem{nagy}
Z.~L. Nagy.
\newblock A multipartite version of the {T}ur{\'a}n problem-density conditions
  and eigenvalues.
\newblock {\em The Electronic Journal of Combinatorics}, 18(1):P46, 2011.

\bibitem{petreolle2018lattice}
M.~P{\'e}tr{\'e}olle, A.~Sokal, and B.~Zhu.
\newblock Lattice paths and branched continued fractions: An infinite sequence
  of generalizations of the {S}tieltjes--{R}ogers and {T}hron--{R}ogers
  polynomials, with coefficientwise {H}ankel-total positivity.
\newblock {\em arXiv preprint arXiv:1807.03271}, 2018.

\bibitem{Tao}
T.~Tao.
\newblock Moser's entropy compression argument.
\newblock {\em
  https://terrytao.wordpress.com/2009/\\08/5/mosers-entropy-compression-argument},
  2009.

\end{thebibliography}
\bibliographystyle{plain}
\end{document}